\documentclass[12pt]{amsart}
\usepackage{amsfonts, amssymb, latexsym, epsfig, epic}
\usepackage{wrapfig}

\newtheorem{theorem}{Theorem}[section]
\newtheorem{proposition}[theorem]{Proposition}

\newtheorem*{claim*}{Claim}
\newtheorem{corollary}[theorem]{Corollary}
\newtheorem{Main Conjecture}[theorem]{Main Conjecture}
\newtheorem{conjecture}[theorem]{Conjecture}
\theoremstyle{remark}
\newtheorem{definition}[theorem]{Definition}
\newtheorem{example}[theorem]{Example}

\newtheorem{remark}[theorem]{Remark}
\theoremstyle{plain}

\newcommand{\excise}[1]{}%{$\star$\textsc{#1}$\star$}

\newcommand{\cellsize}{12}
\newlength{\cellsz} 
\newsavebox{\cell}
\sbox{\cell}{\begin{picture}(\cellsize,\cellsize)
\put(0,0){\line(1,0){\cellsize}}
\put(0,0){\line(0,1){\cellsize}}
\put(\cellsize,0){\line(0,1){\cellsize}}
\put(0,\cellsize){\line(1,0){\cellsize}}
\end{picture}}
\newcommand\cellify[1]{\def\thearg{#1}\def\nothing{}%
\ifx\thearg\nothing
\vrule width0pt height\cellsz depth0pt\else
\hbox to 0pt{\usebox{\cell} \hss}\fi%
\vbox to \cellsz{
\vss
\hbox to \cellsz{\hss$#1$\hss}
\vss}}
\newcommand\tableau[1]{\vtop{\let\\\cr
\baselineskip -16000pt \lineskiplimit 16000pt \lineskip 0pt
\ialign{&\cellify{##}\cr#1\crcr}}}

\begin{document}
\pagestyle{plain}

\mbox{}
\title{Tropicalization, symmetric polynomials, and complexity}
%\title{Tropicalization, symmetric polynomials, and Grigoriev-Koshevoy's complex%ity conjecture}
\date{October 9, 2017}
\author[Alexander Woo]{Alexander Woo}
\address{Department of Mathematics, 
University of Idaho, Moscow, ID 83844}
\email{awoo@uidaho.edu}

\author{Alexander Yong}

\address{Department of Mathematics, University of Illinois at Urbana-Champaign, Urbana, IL 61801} 
\email{ayong@uiuc.edu}
\date{\today}

\maketitle

\begin{abstract}
D.~Grigoriev-G.~Koshevoy recently proved that tropical 
Schur polynomials have (at worst) polynomial tropical semiring complexity. They also conjectured tropical skew Schur polynomials have at least exponential complexity; we establish a polynomial complexity upper bound.
Our proof uses results about (stable) Schubert polynomials,
due to R.~P.~Stanley and S.~Billey-W.~Jockusch-R.~P.~Stanley, together with a sufficient condition for polynomial complexity that is connected to the saturated Newton polytope property.
\end{abstract}

\section{Introduction}

The {\bf tropicalization} of a polynomial 
\[f=\sum_{(i_1,i_2,\ldots,i_n)\in {\mathbb Z}_{\geq 0}^n} c_{i_1,\ldots,i_n}x_1^{i_1} x_2^{i_2}\cdots x_n^{i_n}
\in {\mathbb C}[x_1,x_2,\ldots,x_n]\]
(with respect to the trivial valuation ${\sf val}(a)=0$ for all $a\in {\mathbb C}^*$)
is defined to be 
\begin{equation}
\label{eqn:thetrop}
{\sf Trop}(f):=\max_{i_1,\ldots,i_n} \{i_1 x_1+ \cdots + i_n x_n
\}.
\end{equation}
This is a polynomial 
over the tropical semiring $({\mathbb R},\oplus, \odot)$, where
\[a\oplus b=\max(a,b) \text{\ and
$a\odot b=a+b$}\] 
respectively denote tropical addition and multiplication, respectively. We refer to the books \cite{Mikhal, MacSturm} for more about tropical mathematics.

Let ${\sf Sym}_n$ denote the ring of symmetric polynomials 
in $x_1,\ldots,x_n$.
A linear basis of ${\sf Sym}_n$ is given by the \emph{Schur
polynomials}. These polynomials are indexed by partitions $\lambda$
(identified with their Ferrers/Young diagrams). They are generating series over semistandard Young tableaux $T$ of shape
$\lambda$ with entries from $[n]:=\{1,2,\ldots,n\}$:
\[s_{\lambda}(x_1,\ldots,x_n):=\sum_{T} x^T
\text{\ \ \ \  where \ \ 
$x^T:=\prod_i x_i^{\text{$\#i$'s in $T$}}$.}\]
The importance of this basis stems from its applications to, for example, enumerative and algebraic combinatorics, the representation theory of symmetric groups and general linear groups, and Schubert calculus on Grassmannians; see, for example, \cite{Fulton,ECII}.

D.~Grigoriev and G.~Koshevoy \cite{Grigoriev} studied
the complexity of the tropical polynomial ${\sf Trop}(s_{\lambda})$ over
$({\mathbb R}, \oplus,\odot)$.  An \emph{arithmetic circuit} is
a circuit where inputs are each labelled by a single variable $x_i$ or a fixed constant,
each gate performs a single $\oplus$ or $\odot$ operation, and there is one
output.  An arithmetic circuit $C$ naturally gives an expression ${\sf res}(C)$, the
tropical polynomial in the variables $x_1,\ldots, x_n$ that it computes.  The
circuit $C$ \emph{evaluates} $f$ if ${\sf res}(C)=f$ as tropical polynomials,
meaning that one can show ${\sf res}(C)=f$ using the
\emph{tropical semiring axioms}, by which we mean 
the semiring axioms along with the idempotence property $a \oplus a = a$.  The \emph{tropical semiring
complexity} of $f$ is the smallest number of gates in a circuit $C$ evaluating
$f$; see \cite[Section~2]{Jerrum}.

The following is  \cite[Theorem~2.1]{Grigoriev}:

\begin{theorem}[D.~Grigoriev-G.~Koshevoy]\label{thm:first}
The tropical semiring complexity of 
${\sf Trop}(s_{\lambda})$ is at most $O(n^2\cdot \lambda_1)$.
\end{theorem}
The \emph{skew-Schur polynomial} 
$s_{\lambda/\mu}(x_1,\ldots,x_n)=\sum_{T} x^T$ 
is the generating series for semstandard tableau of skew shape 
$\lambda/\mu$ with entries from $[n]$.  
When $\mu=\emptyset$ then $s_{\lambda/\emptyset}=s_{\lambda}$; hence skew-Schur polynomials generalize 
Schur polynomials. Also,
\begin{equation}
\label{eqn:skewschurlr}
s_{\lambda/\mu}=\sum_{\nu}c_{\mu,\nu}^{\lambda} s_{\nu},
\end{equation}
where $c_{\lambda,\mu}^{\nu}\in {\mathbb Z}_{\geq 0}$ is the \emph{Littlewood-Richardson coefficient}. 
%In \emph{loc.~cit.}, this hardness result was cited as a rationale for
%Conjecture~\ref{conj:main}. 
The next statement is from \cite[Section~5]{Grigoriev}:

\begin{conjecture}[D.~Grigoriev-G.~Koshevoy]\label{conj:main}
The tropical semiring complexity of  ${\sf Trop}(s_{\lambda/\mu})$ is at least exponential.
\end{conjecture}

We will show the following:

\begin{theorem}\label{thm:mainrestate}
There is an explicitly described $\beta$, depending on $\lambda/\mu$, with 
$\beta_1=\lambda_1$, such that
\[{\sf Trop}(s_{\lambda/\mu}(x_1,\ldots,x_n))={\sf Trop}(s_{\beta}(x_1,\ldots,x_n))\]
over the tropical semiring $({\mathbb R},\oplus,\odot)$. 
\end{theorem}

\begin{example}
Let $\lambda=(2,1)$ and $\mu=(1)$. Then the tableaux
contibuting to $s_{\lambda/\mu}$ are:
\[\tableau{&1\\1},\  \tableau{&1\\2},\ \tableau{&2\\1},
\ \tableau{&2\\2}.\]
Hence 
$s_{\lambda/\mu}=x_1^2+2x_1 x_2 +x_2^2$.
On the other hand, (\ref{eqn:skewschurlr}) in this case is:
\[s_{\lambda/\mu}=s_{1,1}+s_{2}=(x_1x_2)+(x_1x_2 +x_1^2+x_2^2).\]

By definition,
\begin{align*}
{\sf Trop}(s_{\lambda/\mu})
& =\max\{x_1+x_2, x_1+x_2, 2x_1, 2x_2\}\\
& = x_1\odot x_2 \oplus x_1\odot x_2 \oplus x_1^{\odot 2}
\oplus x_2^{\odot 2}\\
&= x_1\odot x_2 \oplus x_1^{\odot 2} 
\oplus x_2^{\odot 2} \text{ \ \ (idempotence)}\\
&={\sf Trop}(s_{2}),
\end{align*}
in agreement with Theorem~\ref{thm:mainrestate}. \qed
\end{example}

The following addresses 
Conjecture~\ref{conj:main}:

\begin{corollary}[of Theorems~\ref{thm:first} and~\ref{thm:mainrestate}]
\label{thm:main}
${\sf Trop}(s_{\lambda/\mu})$ has at most
$O(n^2\cdot \lambda_1)$ tropical semiring complexity.
\end{corollary}

In Section~2, we describe a sufficient condition for polynomial complexity. This is explained in terms of \emph{saturated Newton
polytopes} \cite{MTY}.
Section~3 applies this condition to \emph{Stanley symmetric polynomials} \cite{Stanley84}. Since skew-Schur polynomials are a special case of Stanley symmetric
polynomials, we thereby deduce Theorem~\ref{thm:mainrestate}. In Section~4, we remark on how the condition applies to other families of symmetric polynomials. 

\section{Dominance order, Newton polytopes and saturation}

	Let ${\sf Par}(d)=\{\lambda:\lambda\vdash d\}$ be the set of partitions of size $d$. \emph{Dominance order} $\leq_D$ on 
${\sf Par}(d)$ is defined by
\begin{equation}
\label{eqn:domorderineq}
\mu\leq_{D} \lambda \text{\ \ \ if \ \ \ $\sum_{i=1}^k\mu_i\leq \sum_{i=1}^k\lambda_i$ \ \ \ for all $k\geq 1$.}
\end{equation}

\begin{definition}
\label{def:dominated}
Suppose $f\in {\sf Sym}_n$ is homogeneous of degree $d$ such that 
\[f = \sum_{\mu\in {\sf Par}(d)} c_\mu s_\mu\]
with $c_{\mu}\geq 0$ for all $\mu$. Moreover, assume 
there exists $\lambda$ with 
$c_{\lambda}\neq 0$
such that $c_\mu \neq 0$ only if 
$\mu \leq_D \lambda$.  Then we say $f$ is \emph{dominated} by $s_{\lambda}$. \qed
\end{definition}

The \emph{Newton polytope} of 
a polynomial $f$ is 
the convex hull of its exponent vectors, so 
\[{\sf Newton}(f)={\sf conv}((i_1,i_2,\ldots,i_n):c_{i_1,i_2,\ldots,i_n}\neq 0\})\subseteq {\mathbb R}^n.\]
C.~Monical, N.~Tokcan and the second author
\cite{MTY} define $f$ to have  \emph{saturated Newton polytope} (SNP) 
if $c_{i_1,\ldots,i_n}\neq 0 \text{\  whenever $(i_1,\ldots,i_n)\in {\sf Newton}(f)$.}$
 
The \emph{permutahedron} of $\lambda=(\lambda_1,\ldots,\lambda_n)$, denoted ${\mathcal P}_{\lambda}$, is the 
convex hull of the $S_n$-orbit of $\lambda$ in  ${\mathbb R}^n$.
It follows from R.~Rado's theorem \cite{rado} that 
\[{\sf Newton}(s_{\lambda})={\mathcal P}_{\lambda}
\text{ \ \ \ \ and \ \ \ \ 
${\sf Newton}(s_{\mu})\subseteq {\sf Newton}(s_{\lambda})$
 \ if \
$\mu\leq_D \lambda$.}\]

A consequence of R.~Rado's theorem \cite[Proposition~2.5]{MTY} is therefore:

\begin{proposition}
\label{prop:xzy}
If $f\in {\sf Sym}_n$ is dominated by $s_{\lambda}$, then $f$ is SNP, and
\[{\sf Newton}(f)={\sf Newton}(s_{\lambda})={\mathcal P}_{\lambda}.\]
\end{proposition}

We give a technical strengthening of \cite[Theorem~2.5]{Grigoriev}:
\begin{proposition}[Sufficient condition for polynomial complexity]
\label{prop:strengthening}
Suppose $f\in {\sf Sym}_n$ is dominated by $s_{\lambda}(x_1,\ldots,x_n)$. 
Then ${\sf Trop}(f)={\sf Trop}(s_{\lambda})$ over the tropical 
semiring
$({\mathbb R},\oplus,\odot)$. Therefore,
$f$ has at most $O(n^2\cdot \lambda_1)$ tropical semiring complexity.
\end{proposition}
\begin{proof}
By Proposition~\ref{prop:xzy},
\begin{equation}
\label{eqn:newtonequal}
{\sf Newton}(f)={\sf Newton}(s_{\lambda}),
\end{equation}
and $f$ is SNP. This proves the first statement.

At this point, we can appeal to Theorem~\ref{thm:first} to obtain the
second claim. However, for convenience, we recall the ideas from \cite[Theorem~2.5]{Grigoriev}, thus indicating the underlying circuit. There it is shown
that
\begin{equation}
\label{eqn:theirfact}
{\sf Newton}(s_{\lambda})[{\mathbb Z}]=\sum_{1\leq k\leq \lambda_1} {\sf Newton}(e_{\lambda'_k})[{\mathbb Z}].
\end{equation}
In the Minkowski sum of (\ref{eqn:theirfact}), 
\[e_{k}=\sum_{1\leq j_1<j_2<\ldots <j_k\leq n} x_{j_1}\cdots x_{j_k}\] 
is the \emph{elementary symmetric polynomial} of degree $k$. Also, $\lambda'$ is the conjugate partition of $\lambda$, obtained by
transposing the Young diagram for $\lambda$. Finally, for a 
polytope ${\mathcal P}\subseteq {\mathbb R}^n$, ${\mathcal P}[{\mathbb Z}]$ denotes the set of integer lattice points of ${\mathcal P}$. 

Combining (\ref{eqn:newtonequal}) and (\ref{eqn:theirfact}), we see 
\begin{equation}
\label{eqn:hello999}
{\sf Newton}(f)[{\mathbb Z}]=\sum_{1\leq k\leq \lambda_1} {\sf Newton}(e_{\lambda'_k})[{\mathbb Z}].
\end{equation}
By Proposition~\ref{prop:xzy}, $f$ is SNP. This property of $f$, together with (\ref{eqn:hello999}), implies
\begin{equation}
\label{eqn:sept29abc}
{\sf Trop}(f)=\bigodot_{1\leq k\leq \lambda_1} {\sf Trop}(e_{\lambda_k'}),
\end{equation}
as tropical polynomials.

Therefore, following \emph{loc.~cit.}, to calculate ${\sf Trop}(f)$ it suffices to compute ${\sf Trop}(e_{\lambda_k'})$
for $1\leq k\leq \lambda_1$. The latter has at worst $O(n^2)$ complexity, using the 
(tropicalization) of the
Pascal-type recurrence
\[e_k(x_1,\ldots,x_n)=e_{k}(x_1,\ldots,x_{n-1})+x_n e_{k-1}(x_1,\ldots,x_{n-1}).\]
This proves the second claim. 
\end{proof}

\begin{remark}
The assumption in Definition~\ref{def:dominated} that $f$ be Schur-positive
($c_{\mu}\geq 0$ for each $\mu$) is needed for Proposition~\ref{prop:xzy}. Consider the \emph{monomial symmetric polynomial}
$m_{\lambda}:=\sum_{\theta} x_1^{\theta_1}\cdots x_n^{\theta_n}$,
where the sum is over distinct rearrangements of $\lambda$. It is true that
$m_{\lambda}=\sum_{\mu\leq_D\lambda} I_{\lambda,\mu}
s_{\mu}$,
where $I_{\lambda,\lambda}=1$. Yet, $\{m_{\lambda}\}$ has exponential complexity, by
\cite{Grigoriev}.\qed
\end{remark}

\section{Stanley symmetric polynomials and the Proof of Theorem~\ref{thm:mainrestate}}

For any permutation $w$, R.~P.~Stanley \cite{Stanley84} defined the symmetric power series
\[F_w=\sum_{{\bf a}\in {\sf Red}(w)} \sum_{{\bf b}\in C({\bf a})} x_{{\bf b}}.\]
Here ${\sf Red}(w)$ is the set of reduced words for $w$ in the simple transpositions $s_i=(i \ i+1)$. This means
${\bf a}=(a_1,a_2,\ldots,a_\ell)$, where 
$s_{a_1}s_{a_2}\cdots s_{a_\ell}=w$ 
and $\ell=\ell(w)$ is the number of inversions of $w$.
Now, if ${\bf b}=(b_1,\ldots,b_\ell)$, then ${\bf b}\in C({\bf a})$ if
\begin{itemize}
\item $1\leq b_1\leq b_2\leq \cdots \leq b_{\ell}$; and
\item $a_i<a_{i+1}\implies b_i<b_{i+1}$.
\end{itemize}
Set $x_{\bf b}:=x_{b_1}x_{b_2}\cdots x_{b_{\ell}}$. (This actually defines
$F_{w^{-1}}$ in \cite{Stanley84}. Thus we use the results of \emph{loc.~cit.}
with this swap of convention.)

\begin{remark}
The original motivation for $F_w$ is that
$\#{\sf Red}(w)=[x_1 x_2\cdots x_{\ell}]F_w$. 
If we define $a_{w\lambda}$ as the coefficients in
\begin{equation}
\label{eqn:stanpos}
F_w=\sum_{\lambda} a_{w\lambda}s_{\lambda},
\end{equation}
then 
$a_{w\lambda}\in {\mathbb Z}_{\geq 0}$.
This nonnegativity is proved by work of \cite{EdelmanGreene}
(see also \cite{LS}). In fact, $a_{w\lambda}$ is a generalization of the Littlewood-Richardson coefficient. 
A theorem of H.~Narayanan \cite{Narayanan} states that computation of
$c_{\lambda,\mu}^{\nu}$ is $\#{\sf P}$-complete in L.~Valiant's 
complexity theory for counting problems \cite{Valiant}. 
Hence
$a_{w\lambda}$ is a $\#{\sf P}$-complete counting problem.
In particular, this means that there is no polynomial time algorithm for computing either 
$c_{\lambda,\mu}^{\nu}$ or $a_{w\lambda}$ unless ${\sf P}={\sf NP}$.

Now, 
$[x_1\ldots x_\ell]s_{\lambda}=f^{\lambda}$ counts standard Young tableaux of shape $\lambda$. These numbers are computed by the famous \emph{hook-length formula}. The resulting enumeration
$\#{\sf Red}(w)=\sum_{\lambda} a_{w\lambda} f^\lambda$
establishes that
$\#{\sf Red}(w)$ is a $\#{\sf P}$ counting problem. Is it $\#P$-complete?
\qed
\end{remark}

Recall that the \emph{Rothe diagram}
of $w$ is given by 
\[D(w)=\{(i,j): 1\leq i,j\leq n, j<w(i), i<w^{-1}(j)\}.\]
Pictorially, this is described by
placing $\bullet$'s in positions $(i,w(i))$ (in matrix notation), striking out boxes below and to the right of each $\bullet$. Then $D(w)$ consists of the remaining boxes.

For example, if 
\begin{equation}
\label{eqn:theperm}
w=4 \ 1 \ 5 \ 2 \ 7 \ 3 \ 9 \ 6 \ 10 \ 8\in S_{10}
\text{ \ \ \ \ (in one line notation),}
\end{equation}
then $D(w)$ is depicted by: 
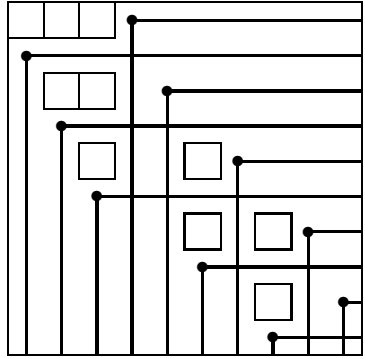
\[
\begin{picture}(273.33,140)
\put(53.33,0){\makebox[0pt][l]{\framebox(133.33,133.33)}}
\put(100,126.67){\circle*{4}}
\thicklines
\put(100,126.67){\line(1,0){86.67}}
\put(100,126.67){\line(0,-1){126.67}}

\put(60,113.33){\circle*{4}}
\thicklines
\put(60,113.33){\line(1,0){126.66}}
\put(60,113.33){\line(0,-1){113.33}}
\put(113.33,100){\circle*{4}}
\put(113.33,100){\line(1,0){73.33}}
\put(113.33,100){\line(0,-1){100}}
\put(73.33,86.67){\circle*{4}}
\put(73.33,86.67){\line(1,0){113.33}}
\put(73.33,86.67){\line(0,-1){86.67}}

\put(140,73.33){\circle*{4}}
\put(140,73.33){\line(1,0){46.67}}
\put(140,73.33){\line(0,-1){73.33}}

\put(86.67,60){\circle*{4}}
\put(86.67,60){\line(1,0){100}}
\put(86.67,60){\line(0,-1){60}}

\put(166.67,46.67){\circle*{4}}
\put(166.67,46.67){\line(1,0){20}}
\put(166.67,46.67){\line(0,-1){46.67}}

\put(126.66,33.33){\circle*{4}}
\put(126.66,33.33){\line(1,0){60}}
\put(126.66,33.33){\line(0,-1){33.33}}

\put(180,20){\circle*{4}}
\put(180,20){\line(1,0){6.67}}
\put(180,20){\line(0,-1){20}}

\put(153.33,6.67){\circle*{4}}
\put(153.33,6.67){\line(1,0){33.33}}
\put(153.33,6.67){\line(0,-1){6.67}}

\thinlines

\put(53.33,120){\makebox[0pt][l]{\framebox(40,13.33)}}
\put(66.67,133.33){\line(0,-1){13.33}}
\put(80,133.33){\line(0,-1){13.33}}

\put(66.67, 93.33){\makebox[0pt][l]{\framebox(26.67,13.33)}}
\put(80,106.67){\line(0,-1){13.33}}

\put(80, 66.67){\makebox[0pt][l]{\framebox(13.33,13.33)}}

\put(120, 66.67){\makebox[0pt][l]{\framebox(13.33,13.33)}}

\put(120, 40){\makebox[0pt][l]{\framebox(13.33,13.33)}}

\put(146.67, 40){\makebox[0pt][l]{\framebox(13.33,13.33)}}

\put(146.67, 13.33){\makebox[0pt][l]{\framebox(13.33,13.33)}}

\end{picture}
\]

For $w\in S_m$, set
$q_i$ to be the number of boxes of $D(w)$ in column $i$ (counting from the left)
for $1\leq i\leq m$.
Then $(q_1,q_2,\ldots,q_m)$ is the \emph{code} of $w^{-1}$.
Let $\beta_{\rm max}(w)$ be the partition obtained by sorting $(q_1,q_2,\ldots,q_m)$ in decreasing order and taking the conjugate shape.

\begin{theorem}[Complexity of tropical Stanley polynomials]
\label{thm:standom}
Let $w\in S_m$. Then the tropical semiring complexity of
${\sf Trop}(F_w(x_1,\ldots,x_n))$ is at most $O(n^2 \cdot \beta_{\rm max}(w)_1)$.
\end{theorem}
\begin{proof}
% Computing $\beta_{\rm max}(w)$ from $w\in S_m$ takes $O(m^2)$ steps.
By \cite[Theorem~4.1]{Stanley84} (up to convention), 
if $a_{w,\lambda}\neq 0$, then 
\[\lambda\leq_D \beta_{\rm max}(w).\]
%By definition of $D(w)$, $\max_i q_i<m$, hence $s_{\beta_{\rm max}(w)}(x_1,\ldo%ts,x_m,0,0,\ldots)$ is not identically zero. 
Since the $a_{w\lambda}$ in (\ref{eqn:stanpos}) are positive,
if $F_w(x_1,\ldots,x_n)$ is nonzero, then $a_{w,\beta_{\rm max}(w)}\neq 0$ and $F_w$ is dominated by $s_{\beta_{\rm max}(w)}$. 
Now use
Proposition~\ref{prop:strengthening}.
\end{proof}

\noindent\emph{Proof of Theorem~\ref{thm:mainrestate} (and Corollary~\ref{thm:main}):} We show that $s_{\lambda/\mu}(x_1,\ldots,x_n)$ is dominated by
$s_{\beta}(x_1,\ldots,x_n)$ for some shape $\beta$ (to be determined) with $\beta_1=\lambda_1$.

Given $\lambda/\mu$, construct a permutation $w_{\lambda/\mu}$ by filling all boxes in the same northwest-southeast diagonal with the same entry, starting with $1$ on the northeastmost diagonal and increasing consecutively as one moves southwest.
Call this filling $T_{\lambda/\mu}$.

For instance, if $\lambda/\mu=(5,4,3,2,1)/(2,2,1,0,0)$ then
\[T_{\lambda/\mu}=\tableau{&& 3 & 2 & 1\\ && 4 & 3\\
& 6 & 5\\ 8 & 7 \\ 9}\]

Let $(r_1,r_2,\ldots,r_{|\lambda/\mu|})$ be the left-to-right, top-to-bottom, row reading word of $T_{\lambda/\mu}$. In our example, this is $(3,2,1,4,3,6,5,8,7,9)$.

Define 
$w_{\lambda/\mu}=s_{r_1}s_{r_2}\cdots s_{r_{|\lambda/\mu|}}$. 
By \cite[Corollary~2.4]{BJS},
\[F_{w_{\lambda/\mu}}(x_1,\ldots,x_n)=s_{\lambda/\mu}(x_1,\ldots,x_n)\in {\sf Sym}_n.\]
%(This conclusion is well-known. For instance, a generalization to \emph{Grothen%dieck polynomials} is discussed in \cite[Section~2]{Buch}.)

By \cite[Section~2]{BJS}, $\lambda/\mu$ is obtained by removing empty rows and columns
of $D(w_{\lambda/\mu})$ and reflecting across a vertical line. In our example, $w_{\lambda/\mu}$ is the permutation
(\ref{eqn:theperm}). The reader can check from the Rothe diagram that this process gives $\lambda/\mu$.
%\[\tableau{\ & \ & \ \\ & \ & \  \\ && \ & \ \\ &&& \ & \ \\ &&&& \ %} \ \ \ \ \ \mapsto \ \ \ \ \ \tableau{
%&&& \ & \ \\ && \ & \ \\ \ & \ &  \ \\ \ & \ \\ \ }=\lambda/\mu.\]

By definition, $\beta_{\rm max}(w_{\lambda/\mu})$ is the conjugate of the decreasing rearrangement of the code of $w_{\lambda/\mu}^{-1}$. 
Hence, in our example, the code of $w_{\lambda/\mu}^{-1}$ is $(1,2,3,0,0,2,0,2,0,0)$, which rearranges to $(3,2,2,2,1,0,0,0,0,0)$. Therefore,
\[\beta_{\rm max}(w_{\lambda/\mu})=\tableau{\ & \ & \ \\ \ & \ \\ \ & \ \\ \ & \ \\ \ }^{\ '}=
\tableau{\ & \ & \ & \ & \ \\ \ & \ & \ & \ \\ \ }.\]
(Thus, $\beta_{\rm max}(w_{\lambda/\mu})$ is obtained from $\lambda/\mu$ by first pushing the boxes in each column north and left-justifying the result.)

Since the coefficients $c_{\mu,\nu}^\lambda$ in the Schur expansion (\ref{eqn:skewschurlr}) are positive, 
$s_{\lambda/\mu}(x_1,\ldots,x_n)$ is dominated by
$s_{\beta}$, where $\beta=\beta_{\rm max}(w_{\lambda/\mu})$. By the above process from~\cite{BJS} relating $D(w_{\lambda/\mu})$ and $\lambda/\mu$,
\[\beta_1=\beta_{\rm max}(w_{\lambda/\mu})_1=\lambda_1,\]
as desired. Theorem~\ref{thm:mainrestate} now holds by
the first conclusion of Proposition~\ref{prop:strengthening}. 

To conclude Corollary~\ref{thm:main}, we may now either apply 
Theorem~\ref{thm:first} or Theorem~\ref{thm:standom}.\qed

\section{Some other symmetric polynomials}

In \cite[Sections~2 and~3]{MTY}, some symmetric polynomials are 
observed to be SNP because they are 
dominated by $s_{\lambda}$ (for some $\lambda$), or for other reasons. These include:
\begin{itemize}
\item[(1)] J.~R.~Stembridge's polynomial $F_M$ for a totally nonnegative matrix $M$ \cite{Stembridge};
\item[(2)] the cycle index polynomial $c_G$ of  a subgroup $G\leqslant S_n$ (from Redfield-P\'olya theory);
\item[(3)] C.~Reutenauer's $q_{\lambda}$ basis of ${\sf Sym}_n$ \cite{Reutenauer};
\item[(4)] the symmetric Macdonald polynomial (where $(q,t)\in {\mathbb C}^2$ is generic);
\item[(5)] the Hall-Littlewood polynomial (for any positive evaluation of $t$); and
\item[(6)] Schur $P-$ and Schur $Q-$ polynomials.
\end{itemize}

Consequently, by Proposition~\ref{prop:strengthening}, the tropicalizations of these polynomials equal some tropical Schur polynomial. Therefore, as with skew Schur polynomials, one obtains immediate tropical semiring complexity implications:
\begin{itemize}
\item The polynomials (1) and (2) are dominated by $s_{(n)}$;
  see~\cite[Theorem~2.28]{MTY} and~\cite[Theorem~2.30]{MTY}. Hence
  Proposition~\ref{prop:strengthening} shows their tropicalizations
  have $O(n^3)$ complexity.  In fact, since ${\sf Trop}(s_{(n)})={\sf
    Trop}((x_1+\cdots+x_n)^n)$ as tropical polynomials, they have $O(n)$ complexity (see~\cite[Theorem~1.4]{Fomin} for a nontropical version of this
statement).

\item For (3), if $\lambda=(\lambda_1,\ldots,\lambda_{\ell},1^r)$
where each $\lambda_i\geq 2$, then $q_{\lambda}$ is dominated
by $s_{a,b}$ where $a=|\lambda|-\ell$ and $b=\ell$ (see \cite[Theorem~2.3.2]{MTY} and specifically its proof). Hence
Proposition~\ref{prop:strengthening} asserts ${\sf Trop}(q_{\lambda})$ has $O(n^2\cdot (|\lambda|-\ell))$ complexity.

\item For a generic choice of $q,t\in {\mathbb C}^2$, it follows from \cite[Section~3.1]{MTY} that if $P_{\lambda}(X;q,t)\in {\sf Sym}_n$ is the Macdonald polynomial, then
${\sf Trop}(P_{\lambda}(X;q,t))={\sf Trop}(s_{\lambda})$.
Hence ${\sf Trop}(P_{\lambda}(X;q,t))$ has $O(n^2\cdot\lambda_1)$ complexity.
\item (5) and (6) are also indexed by
partitions $\lambda$ and dominated by $s_{\lambda}$. Thus,
Proposition~\ref{prop:strengthening} implies their tropicalizations have $O(n^2\cdot \lambda_1)$ complexity.
\end{itemize}

\section*{Acknowledgements}
AY was supported by an NSF grant. We thank Marc Snir for helpful communications.

\end{document}